\documentclass[10pt]{ieeeconf}

\IEEEoverridecommandlockouts

\usepackage{mathptmx}
\usepackage{times}
\usepackage{amsmath}
\usepackage{amssymb}

\newtheorem{theorem}{Theorem}

\newtheorem{definition}[theorem]{Definition}
\newtheorem{example}[theorem]{Example}
\newtheorem{problem}[theorem]{Problem}
\newtheorem{remark}[theorem]{Remark}

\title{\LARGE \bf Variable Order Fractional Variational Calculus\\
for Double Integrals$^*$}

\author{Tatiana Odzijewicz$^{1}$, Agnieszka B. Malinowska$^{2}$ and Delfim F. M. Torres$^{1}$%
\thanks{*Part of the first author's Ph.D., which is carried out at the University of Aveiro
under the Doctoral Programme \emph{Mathematics and Applications} of Universities of Aveiro and Minho.}%
\thanks{$^{1}$T. Odzijewicz and D. F. M. Torres are with the Center for Research
and Development in Mathematics and Applications, Department of Mathematics, University of Aveiro,
3810-193 Aveiro, Portugal {\tt\small tatianao@ua.pt, delfim@ua.pt}}%
\thanks{$^{2}$A. B. Malinowska is with the Faculty of Computer Science,
Bia{\l}ystok University of Technology, 15-351 Bia\l ystok, Poland
{\tt\small a.malinowska@pb.edu.pl}}}

\begin{document}

\maketitle


\begin{abstract}
We introduce three types of partial fractional operators of variable order.
An integration by parts formula for partial fractional integrals of variable order
and an extension of Green's theorem are proved. These results allow us to obtain
a fractional Euler--Lagrange necessary optimality condition for variable
order two-dimensional fractional variational problems.
\end{abstract}




\section{INTRODUCTION}

Fractional variational calculus is a mathematical discipline that consists
in extremizing (minimizing or maximizing) functionals whose Lagrangians
contain fractional integrals and derivatives. For the first link between
calculus of variations and fractional calculus we should look back
to the XIXth century. In 1823, Niels Heinrik Abel considered the problem
(Abel's mechanical problem) of finding a curve, lying in a vertical plane,
for which the time taken by a material point sliding without friction
from the highest point to the lowest one, is destined function of height \cite{Abel}.
Abel's mechanical problem is a generalization of the tautochrone problem,
which is part of the calculus of variations (and optimal control).
Despite of this early example, fractional variational calculus became
a research field only in the XXth century.
The subject was initiated in 1996-1997 by Riewe, who derived Euler--Lagrange fractional
differential equations and showed how non-conservative systems
in mechanics can be described using fractional derivatives
\cite{CD:Riewe:1996,CD:Riewe:1997}. Nowadays, the fractional calculus
of variations and fractional optimal control are strongly developed (see, \textrm{e.g.},
\cite{MR2762638,MyID:209,Shakoor:01,Cresson,gastao1,MyID:089,MR2905862,book:Klimek,MyID:181,MyID:207,FVC_Gen,MyID:227}).
For the state of the art, we refer the reader to the recent book \cite{MyID:208}.

In 1993, Samko and Ross investigated integrals and derivatives not of a constant
but of variable order \cite{Samko:Ross1,Samko,Samko:Ross2}. Afterwards, several pure mathematical
and applicational papers contributed to the theory of variable order fractional calculus
(see, \textrm{e.g.}, \cite{Atanackovic1,Coimbra,Diaz:Coimbra,Lorentzo,Ramirez:Coimbra1,Ramirez:Coimbra2}).
Here, our primary goal is to study problems of the calculus of variations
with functionals given by two-dimensional definite integrals
involving partial derivatives of variable fractional order.
It should be mentioned that most results in fractional variational calculus
are for single time, and that the literature regarding the multidimensional
case is scarce: in \cite{MyID:182} a fractional theory of the calculus of variations for multiple integrals
is developed for Riemann--Liouville fractional derivatives
and integrals in the sense of Jumarie; in \cite{Cresson2}
a Lagrangian structure for the Stokes equation, the fractional wave equation, the diffusion
or fractional diffusion equations, are obtained using a fractional embedding theory;
and in \cite{tatiana} fractional isoperimetric problems of calculus of
variations with double integrals are considered. Here we develop
a more general fractional theory of the calculus of variations for multiple integrals,
where the fractional order is not a constant but a function.

The article is organized as follows. In Section~\ref{sec:prelim} we give the definitions and basic properties
of both ordinary and partial integrals and derivatives of variable fractional order. An extension of Green's theorem,
to the variable fractional order, is then obtained in Section~\ref{sec:MR}. Section~\ref{sec:VOFVC} gives the proof
of a necessary optimality condition for the two-dimensional fundamental problem of the calculus of variations.
We finish with Section~\ref{sec:conc} of conclusions.


\section{VARIABLE ORDER FRACTIONAL OPERATORS}
\label{sec:prelim}

In this section we introduce the notions of ordinary
and partial fractional operators of variable order.
Along the text $L_1$ denotes the class of Lebesgue
integrable functions, $AC$ the class of absolutely
continuous functions, and by $\partial_i F$ we understand
the partial derivative of a certain function $F$ with respect to
its $i$th argument.

\begin{definition}
\label{def:VORLI}
Let $0<\alpha(t,\tau)<1$ for all $t, \tau \in [a,b]$, $f\in L_1 [a,b]$,
and $\Gamma$ be the Gamma function, \textrm{i.e.},
$$
\Gamma(z) = \displaystyle \int_0^\infty \mathrm{e}^{-t} t^{z-1} dt.
$$
Then,
\begin{displaymath}
{_{a}}\textsl{I}^{\alpha(\cdot,\cdot)}_{t}f(t)
= \int\limits_a^t\frac{1}{\Gamma(\alpha(t,\tau))}
(t-\tau)^{\alpha(t,\tau)-1}f(\tau)d\tau \quad (t>a)
\end{displaymath}
is called the left Riemann--Liouville integral
of variable fractional order $\alpha(\cdot,\cdot)$, while
\begin{displaymath}
{_{t}}\textsl{I}^{\alpha(\cdot,\cdot)}_{b}f(t)
=\int\limits_t^b \frac{1}{\Gamma(\alpha(\tau,t))}
(\tau-t)^{\alpha(\tau,t)-1}f(\tau)d\tau \quad (t<b)
\end{displaymath}
denotes the right Riemann--Liouville integral
of variable fractional order $\alpha(\cdot,\cdot)$.
\end{definition}

\begin{example}[\cite{Samko:Ross2}]
Let $\alpha(t,\tau) = \alpha(t)$ be a function
depending only on variable $t$,
$\frac{1}{n}<\alpha(t)<1$ for all $t \in [a,b]$
and a certain $n\in\mathbb{N}$ greater or equal than two,
and $\gamma>-1$. Then,
\begin{equation*}
{_{a}}\textsl{I}^{\alpha(\cdot)}_{t} (t-a)^{\gamma}
=\frac{\Gamma(\gamma+1)(t-a)^{\gamma+\alpha(t)}}{\Gamma(\gamma+\alpha(t)+1)}.
\end{equation*}
\end{example}

\begin{definition}
\label{def:VORLD}
Let $0<\alpha(t,\tau)<1$ for all $t, \tau \in [a,b]$.
If ${_{a}}\textsl{I}^{1-\alpha(\cdot,\cdot)}_{t}f \in AC[a,b]$,
then the left Riemann--Liouville derivative of variable fractional
order $\alpha(\cdot,\cdot)$ is defined by
\begin{multline*}
{_{a}}\textsl{D}^{\alpha(\cdot,\cdot)}_{t} f(t)
= \frac{d}{dt} {_{a}}\textsl{I}^{1-\alpha(\cdot,\cdot)}_{t} f(t)\\
=\frac{d}{dt}\int\limits_a^t
\frac{1}{\Gamma(1-\alpha(t,\tau))}(t-\tau)^{-\alpha(t,\tau)} f(\tau)d\tau \quad (t>a),
\end{multline*}
while the right Riemann--Liouville derivative of variable fractional order
$\alpha(\cdot,\cdot)$ is defined for functions $f$ such that
${_{t}}\textsl{I}^{1-\alpha(\cdot,\cdot)}_{b}f\in AC[a,b]$ by
\begin{multline*}
{_{t}}\textsl{D}^{\alpha(\cdot,\cdot)}_{b}f(t)
= -\frac{d}{dt} {_{t}}\textsl{I}^{1-\alpha(\cdot,\cdot)}_{b}f(t)\\
=\frac{d}{dt}\int\limits_t^b
\frac{-1}{\Gamma(1-\alpha(\tau,t))}(\tau-t)^{-\alpha(\tau,t)}f(\tau)d\tau \quad (t<b).
\end{multline*}
\end{definition}

\begin{definition}
\label{definition:Caputo}
Let $0<\alpha(t,\tau)<1$ for all $t, \tau \in [a,b]$.
If $f\in AC[a,b]$, then the left Caputo derivative
of variable fractional order $\alpha(\cdot,\cdot)$ is defined by
\begin{multline*}
{^{C}_{a}}\textsl{D}^{\alpha(\cdot,\cdot)}_{t}f(t)
={_{a}}\textsl{I}^{1-\alpha(\cdot,\cdot)}_{t}\frac{d}{dt}  f(t)\\
=\int\limits_a^t
\frac{1}{\Gamma(1-\alpha(t,\tau))}(t-\tau)^{-\alpha(t,\tau)}\frac{d}{d\tau}f(\tau)d\tau \quad (t>a),
\end{multline*}
while the right Caputo derivative of variable fractional order $\alpha(\cdot,\cdot)$ is given by
\begin{multline*}
{^{C}_{t}}\textsl{D}^{\alpha(\cdot,\cdot)}_{b}f(t)
=-{_{a}}\textsl{I}^{1-\alpha(\cdot,\cdot)}_{t}\frac{d}{dt}  f(t)\\
=\int\limits_t^b\frac{-1}{\Gamma(1-\alpha(\tau,t))}
(\tau-t)^{-\alpha(\tau,t)}\frac{d}{d\tau}f(\tau)d\tau \quad (t<b).
\end{multline*}
\end{definition}

Let $\Delta_n=[a_1,b_1]\times\dots\times [a_n,b_n]$, $n\in\mathbb{N}$, be a subset
of $\mathbb{R}^n$, $\textbf{t}=(t_1,\dots, t_n) \in \Delta_n$,
and $\alpha_i(\cdot,\cdot):[a_i,b_i]\times [a_i,b_i]\rightarrow \mathbb{R}$ be
such that $0<\alpha_i(t_i,\tau)<1$ for all $t_i, \tau \in [a_i,b_i]$, $i=1,\dots,n$.
Partial integrals and derivatives of variable fractional order are a natural
generalization of the corresponding one-dimensional variable order fractional
integrals and derivatives.

\begin{definition}
\label{def:VOPI}
Let function $f=f(t_1,\dots,t_n)$ be continuous on the set $\Delta_n$.
The left Riemann--Liouville partial integral of variable fractional order
$\alpha_i(\cdot,\cdot)$, with respect to the $i$th variable $t_i$, is given by
\begin{multline*}
{_{a_i}}\textsl{I}^{\alpha_i(\cdot,\cdot)}_{t_i}f(\textbf{t})
= \int\limits_{a_i}^{t_i}\frac{1}{\Gamma(\alpha_i(t_i,\tau))}
(t_i-\tau)^{\alpha_i(t_i,\tau)-1}\\
\times f(t_1,\dots,t_{i-1},\tau,t_{i+1},\dots,t_n)d\tau \quad (t_i>a_i),
\end{multline*}
while
\begin{multline*}
{_{t_i}}\textsl{I}^{\alpha_i(\cdot,\cdot)}_{b_i}f(\textbf{t})
=\int\limits_{t_i}^{b_i} \frac{1}{\Gamma(\alpha_i(\tau,t_i))}
(\tau-t_i)^{\alpha_i(\tau,t_i)-1}\\
\times f(t_1,\dots,t_{i-1},\tau,t_{i+1},\dots,t_n)d\tau \quad (t_i<b_i)
\end{multline*}
denotes the right Riemann--Liouville partial integral of variable
fractional order $\alpha_i(\cdot,\cdot)$ with respect to variable $t_i$.
\end{definition}

\begin{definition}
\label{def:VOPRL}
If ${_{a_i}}\textsl{I}^{1-\alpha_i(\cdot,\cdot)}_{t_i}f \in C^1(\Delta_n)$,
then the left Riemann--Liouville partial derivative of variable fractional
order $\alpha_i(\cdot,\cdot)$, with respect to the $i$th variable $t_i$, is given by
\begin{multline*}
{_{a_i}}\textsl{D}^{\alpha_i(\cdot,\cdot)}_{t_i} f(\textbf{t})
= \frac{\partial}{\partial t_i} {_{a_i}}\textsl{I}^{1-\alpha_i(\cdot,\cdot)}_{t_i} f(\textbf{t})\\
=\frac{\partial}{\partial t_i}\int\limits_{a_i}^{t_i}
\frac{1}{\Gamma(1-\alpha_i(t_i,\tau))}(t_i-\tau)^{-\alpha_i(t_i,\tau)}\\
\times f(t_1,\dots,t_{i-1},\tau,t_{i+1},\dots,t_n) d\tau \quad (t_i>a_i),
\end{multline*}
while the right Riemann--Liouville partial derivative of variable fractional order
$\alpha_i(\cdot,\cdot)$, with respect to the $i$th variable $t_i$,
is defined for functions $f$ such that
${_{t_i}}\textsl{I}^{1-\alpha_i(\cdot,\cdot)}_{b_i}f \in C^1(\Delta_n)$ by
\begin{multline*}
{_{t_i}}\textsl{D}^{\alpha_i(\cdot,\cdot)}_{b_i}f(\textbf{t})
= -\frac{\partial}{\partial t_i} {_{t_i}}\textsl{I}^{1-\alpha_i(\cdot,\cdot)}_{b_i}f(\textbf{t})\\
=\frac{\partial}{\partial t_i}\int\limits_{t_i}^{b_i}
\frac{-1}{\Gamma(1-\alpha_i(\tau,t_i))}(\tau-t_i)^{-\alpha_i(\tau,t_i)}\\
\times f(t_1,\dots,t_{i-1},\tau,t_{i+1},\dots,t_n)d\tau \quad (t_i<b_i).
\end{multline*}
\end{definition}

\begin{definition}
\label{def:VOPC}
Let $f\in C^1(\Delta_n)$. The left Caputo partial derivative
of variable fractional order $\alpha_i(\cdot,\cdot)$,
with respect to the $i$th variable $t_i$, is defined by
\begin{equation*}
\begin{split}
{^{C}_{a_i}}\textsl{D}^{\alpha_i(\cdot,\cdot)}_{t_i} & f(\textbf{t})
= {_{a_i}}\textsl{I}^{1-\alpha_i(\cdot,\cdot)}_{t_i} \frac{\partial}{\partial t_i}f(\textbf{t})\\
&=\int\limits_{a_i}^{t_i}
\frac{1}{\Gamma(1-\alpha_i(t_i,\tau))}(t_i-\tau)^{-\alpha_i(t_i,\tau)}\\
&\quad \times \frac{\partial}{\partial\tau}
f(t_1,\dots,t_{i-1},\tau,t_{i+1},\dots,t_n)d\tau \quad (t_i>a_i),
\end{split}
\end{equation*}
while the right Caputo partial derivative of variable fractional order
$\alpha_i(\cdot,\cdot)$, with respect to the $i$th variable $t_i$, is given by
\begin{equation*}
\begin{split}
{^{C}_{t_i}}\textsl{D}^{\alpha_i(\cdot,\cdot)}_{b_i} & f(\textbf{t})
=-{_{t_i}}\textsl{I}^{1-\alpha_i(\cdot,\cdot)}_{b_i}\frac{\partial}{\partial t_i} f(\textbf{t})\\
&=\int\limits_{t_i}^{b_i}\frac{-1}{\Gamma(1-\alpha_i(\tau,t_i))}
(\tau-t_i)^{-\alpha_i(\tau,t_i)}\\
&\quad \times \frac{\partial}{\partial\tau}
f(t_1,\dots,t_{i-1},\tau,t_{i+1},\dots,t_n)d\tau \quad (t_i<b_i).
\end{split}
\end{equation*}
\end{definition}

\begin{remark}
In Definitions~\ref{def:VOPI}, \ref{def:VOPRL} and \ref{def:VOPC},
all the variables except $t_i$ are kept fixed. That choice of fixed
values determines a function $f_{t_1,\dots,t_{i-1},t_{i+1},\dots,t_n}
:[a_i,b_i]\rightarrow \mathbb{R}$ of one variable $t_i$,
\begin{displaymath}
f_{t_1,\dots,t_{i-1},\dots,t_{i+1},\dots,t_n}(t_i)
=f(t_1,\dots,t_{i-1},t_i,t_{i+1},\dots,t_n).
\end{displaymath}
By Definitions~\ref{def:VORLI}, \ref{def:VORLD}, \ref{definition:Caputo},
\ref{def:VOPI}, \ref{def:VOPRL}, and \ref{def:VOPC}, we have
\begin{multline*}
{_{a_i}}\textsl{I}^{\alpha_i(\cdot,\cdot)}_{t_i}f_{t_1,\dots,t_{i-1},t_{i+1},\dots,t_n}(t_i)\\
={_{a_i}}\textsl{I}^{\alpha_i(\cdot,\cdot)}_{t_i}f(t_1,\dots,t_{i-1},t_i,t_{i+1},\dots,t_n),
\end{multline*}
\begin{multline*}
{_{t_i}}\textsl{I}^{\alpha_i(\cdot,\cdot)}_{b_i}f_{t_1,\dots,t_{i-1},t_{i+1},\dots,t_n}(t_i)\\
={_{t_i}}\textsl{I}^{\alpha_i(\cdot,\cdot)}_{b_i}f(t_1,\dots,t_{i-1},t_i,t_{i+1},\dots,t_n),
\end{multline*}
\begin{multline*}
{_{a_i}}\textsl{D}^{\alpha_i(\cdot,\cdot)}_{t_i}f_{t_1,\dots,t_{i-1},t_{i+1},\dots,t_n}(t_i)\\
={_{a_i}}\textsl{D}^{\alpha_i(\cdot,\cdot)}_{t_i}f(t_1,\dots,t_{i-1},t_i,t_{i+1},\dots,t_n),
\end{multline*}
\begin{multline*}
{_{t_i}}\textsl{D}^{\alpha_i(\cdot,\cdot)}_{b_i}f_{t_1,\dots,t_{i-1},t_{i+1},\dots,t_n}(t_i)\\
={_{t_i}}\textsl{D}^{\alpha_i(\cdot,\cdot)}_{b_i}f(t_1,\dots,t_{i-1},t_i,t_{i+1},\dots,t_n),
\end{multline*}
\begin{multline*}
{^{C}_{a_i}}\textsl{D}^{\alpha_i(\cdot,\cdot)}_{t_i}f_{t_1,\dots,t_{i-1},t_{i+1},\dots,t_n}(t_i)\\
={^{C}_{a_i}}\textsl{D}^{\alpha_i(\cdot,\cdot)}_{t_i}f(t_1,\dots,t_{i-1},t_i,t_{i+1},\dots,t_n),
\end{multline*}
\begin{multline*}
{^{C}_{t_i}}\textsl{D}^{\alpha_i(\cdot,\cdot)}_{b_i}f_{t_1,\dots,t_{i-1},t_{i+1},\dots,t_n}(t_i)\\
={^{C}_{t_i}}\textsl{D}^{\alpha_i(\cdot,\cdot)}_{b_i}f(t_1,\dots,t_{i-1},t_i,t_{i+1},\dots,t_n).
\end{multline*}
Thus, similarly to the integer order case, computation of partial derivatives of variable fractional
order is reduced to the computation of one-variable derivatives of variable fractional order.
\end{remark}

\begin{remark}
If $\alpha_i(\cdot,\cdot)$ is a constant function,
then the partial operators of variable fractional order
are reduced to corresponding partial integrals and derivatives of constant order.
For more information on the classical fractional partial operators
of constant order, we refer to \cite{book:Kilbas,book:Podlubny,book:Samko}.
\end{remark}


\section{GREEN'S THEOREM FOR VARIABLE ORDER FRACTIONAL OPERATORS}
\label{sec:MR}

Green's theorem is useful in many fields of mathematics, physics,
engineering, and fractional calculus \cite{MyID:236}.
We begin by proving a two-dimensional integration by parts formula
for partial integrals of variable fractional order.

\begin{theorem}
\label{theorem:GRI}
Let $\frac{1}{l_i}<\alpha_i(t_i,\tau)<1$ for all $t_i, \tau \in [a_i,b_i]$,
where $l_i\in\mathbb{N}$, $i=1,2$, are greater or equal than two.
If $f,g,\eta_1,\eta_2\in C\left(\Delta_2\right)$, then the partial integrals
of variable fractional order satisfy the following identity:
\begin{multline*}
\int\limits_{a_1}^{b_1}\int\limits_{a_2}^{b_2}\left[
g(\mathbf{t}){_{a_1}}\textsl{I}^{\alpha_1(\cdot,\cdot)}_{t_1}\eta_1(\mathbf{t})
+f(\mathbf{t}){_{a_2}}\textsl{I}^{\alpha_2(\cdot,\cdot)}_{t_2}\eta_2(\mathbf{t})\right]dt_2 dt_1 \\
=\int\limits_{a_1}^{b_1}\int\limits_{a_2}^{b_2} \left[
\eta_1(\mathbf{t}){_{t_1}}\textsl{I}^{\alpha_1(\cdot,\cdot)}_{b_1}g(\mathbf{t})
+\eta_2(\mathbf{t}){_{t_2}}\textsl{I}^{\alpha_2(\cdot,\cdot)}_{b_2}f(\mathbf{t})\right]dt_2 dt_1.
\end{multline*}
\end{theorem}

\begin{proof}
Define
\begin{equation*}
F_1(\mathbf{t},\tau)
:= \begin{cases}{}
\left|\frac{(t_1-\tau)^{\alpha_1(t_1,\tau)-1}}{\Gamma(\alpha_1(t_1,\tau))}
g(\mathbf{t})\eta_1(\tau,t_2)\right|
& \mbox{if $\tau \leq t_1$}\\
0 & \mbox{if $\tau > t_1$}
\end{cases}
\end{equation*}
for all $(\mathbf{t},\tau)\in [a_1,b_1]\times [a_2,b_2]\times [a_1,b_1]$, and
\begin{equation*}
F_2(\mathbf{t},\tau) :=
\begin{cases}
\left|\frac{(t_2-\tau)^{\alpha_2(t_2,\tau)-1}}{\Gamma(\alpha_2(t_2,\tau))}
f(\mathbf{t})\eta_2(t_2,\tau)\right|
& \mbox{if $\tau \leq t_2$}\\
0 & \mbox{if $\tau > t_2$}
\end{cases}
\end{equation*}
for all $(\mathbf{t},\tau)\in [a_1,b_1]\times [a_2,b_2]\times [a_2,b_2]$.
Since $f,g$ and $\eta_i$, $i=1,2$, are continuous functions on $\Delta_2$,
they are bounded on $\Delta_2$, \textrm{i.e.},
there exist positive real numbers $C_1$, $C_2$, $C_3$, $C_4>0$ such that
\begin{displaymath}
\left|f(\mathbf{t})\right|\leq C_1,~\left|g(\mathbf{t})\right|
\leq C_2,~\left|\eta_1(\mathbf{t})\right|
\leq C_3,~\left|\eta_2(\mathbf{t})\right| \leq C_4
\end{displaymath}
for all $\mathbf{t}\in \Delta_2$. Therefore,
\begin{multline*}
\int_{a_1}^{b_1}\Biggl(\int_{a_2}^{b_2}\Biggl(\int_{a_1}^{b_1}
F_1(\mathbf{t},\tau)d\tau+ \int_{a_2}^{b_2} F_2(\mathbf{t},\tau)d\tau\Biggr)dt_2 \Biggr)dt_1\\
=\int_{a_1}^{b_1}\Biggl(\int_{a_2}^{b_2}\Biggl(\int_{a_1}^{t_1}
\Biggl|\frac{(t_1-\tau)^{\alpha_1(t_1,\tau)-1}}{\Gamma(\alpha_1(t_1,\tau))}
g(\mathbf{t})\eta_1(\tau,t_2)\Biggr| d\tau \\
+\int_{a_2}^{t_2} \Biggl|
\frac{(t_2-\tau)^{\alpha_2(t_2,\tau)-1}}{\Gamma(\alpha_2(t_2,\tau))}
f(\mathbf{t})\eta_2(t_1,\tau)\Biggr| d\tau\Biggr)dt_2 \Biggr)dt_1 \\
\leq \int_{a_1}^{b_1}\Biggl(\int_{a_2}^{b_2}\Biggl(C_2 C_3\int_{a_1}^{t_1}
\Biggl|\frac{1}{\Gamma(\alpha_1(t_1,\tau))}(t_1-\tau)^{\alpha_1(t_1,\tau)-1}\Biggr| d\tau \\
+C_1 C_4\int_{a_2}^{t_2} \Biggl|\frac{(t_2-\tau)^{\alpha_2(t_2,\tau)-1}}{\Gamma(\alpha_2(t_2,\tau))}\Biggr|
d\tau\Biggr)dt_2 \Biggr)dt_1.
\end{multline*}
Because $\frac{1}{l_i}<\alpha_i(t_i,\tau)<1$, $i=1,2$,
\begin{enumerate}
\item $\ln(t_i-\tau)\geq 0$  and
$(t_i-\tau)^{\alpha_i(t_i,\tau)-1}<1$
for $t_i-\tau \geq 1$;

\item while $\ln(t_i-\tau)<0$ and
$(t_i-\tau)^{\alpha_i(t_i,\tau)-1}<(t_i-\tau)^{\frac{1}{l_i}-1}$
for $t_i-\tau < 1$.
\end{enumerate}
Therefore,
\begin{multline*}
\int_{a_1}^{b_1}\Biggl(\int_{a_2}^{b_2}\Biggl(C_2 C_3\int_{a_1}^{t_1}
\left|\frac{1}{\Gamma(\alpha_1(t_1,\tau))}(t_1-\tau)^{\alpha_1(t_1,\tau)-1}\right| d\tau \\
+C_1 C_4\int_{a_2}^{t_2} \left|\frac{1}{\Gamma(\alpha_2(t_2,\tau))}(t_2-\tau)^{\alpha_2(t_2,\tau)-1}\right|
d\tau\Biggr)dt_2 \Biggr)dt_1\\
<  \int_{a_1}^{b_1}\Biggl(\int_{a_2}^{b_2}\Biggl(
C_2 C_3\Biggl(\int_{a_1}^{t_1-1}\frac{1}{\Gamma(\alpha_1(t_1,\tau))} d\tau\\
+\int_{t_1-1}^{t_1}\frac{1}{\Gamma(\alpha_1(t_1,\tau))}(t_1-\tau)^{\frac{1}{l_1}-1}d\tau\Biggr)\\
+C_1 C_4\Biggl(\int_{a_2}^{t_2-1}\frac{1}{\Gamma(\alpha_2(t_2,\tau))} d\tau
\\+\int_{t_2-1}^{t_2}\frac{1}{\Gamma(\alpha_2(t_2,\tau))}(t_2
-\tau)^{\frac{1}{l_2}-1}d\tau\Biggr)\Biggr)dt_2\Biggr)dt_1.
\end{multline*}
Moreover, by inequality
\begin{equation*}
\Gamma(x+1) \geq \frac{x^2+1}{x+1},
\end{equation*}
valid for $x \in [0,1]$ (see \cite{Ivady}),
and the property
$$
\Gamma(x+1) = x \Gamma(x)
$$
of the Gamma function, one has
\begin{multline*}
\int_{a_1}^{b_1}\Biggl(\int_{a_2}^{b_2}\Biggl(C_2 C_3\Biggl(
\int_{a_1}^{t_1-1}\frac{1}{\Gamma(\alpha_1(t_1,\tau))} d\tau\\
+\int_{t_1-1}^{t_1}\frac{1}{\Gamma(\alpha_1(t_1,\tau))}(t_1-\tau)^{\frac{1}{l_1}-1}d\tau\Biggr)\\
+C_1 C_4\Biggl(\int_{a_2}^{t_2-1}\frac{1}{\Gamma(\alpha_2(t_2,\tau))} d\tau\\
+\int_{t_2-1}^{t_2}\frac{1}{\Gamma(\alpha_2(t_2,\tau))}(
t_2-\tau)^{\frac{1}{l_2}-1}d\tau\Biggr)\Biggr) dt_2\Biggr) dt_1\\
\leq \int_{a_1}^{b_1}\Biggl(\int_{a_2}^{b_2}\Biggl(C_2 C_3\Biggl(
\int_{a_1}^{t_1-1}\frac{\alpha_1^2(t_1,\tau)+\alpha_1(t_1,\tau)}{\alpha_1^2(t_1,\tau)+1} d\tau\\
+\int_{t_1-1}^{t_1}\frac{\alpha_1^2(t_1,\tau)+\alpha_1(t_1,\tau)}{\alpha_1^2(t_1,\tau)
+1}(t_1-\tau)^{\frac{1}{l_1}-1}d\tau\Biggr)\\
+C_1 C_4\Biggl(\int_{a_2}^{t_2-1}\frac{\alpha_2^2(t_2,\tau)
+\alpha_2(t_2,\tau)}{\alpha_2^2(t_2,\tau)+1} d\tau\\
+\int_{t_2-1}^{t_2}\frac{\alpha_2^2(t_2,\tau)+\alpha_2(t_2,\tau)}{\alpha_2^2(t_2,\tau)
+1}(t_2-\tau)^{\frac{1}{l_2}-1}d\tau\Biggr)\Biggr)dt_2\Biggr)dt_1\\
<\int_{a_1}^{b_1}\Biggl(\int_{a_2}^{b_2}\biggl(C_2 C_3\Biggl(\int_{a_1}^{t_1-1} d\tau
+\int_{t_1-1}^{t_1}(t_1-\tau)^{\frac{1}{l_1}-1}d\tau\Biggr)\\
+C_1 C_4\Biggl(\int_{a_2}^{t_2-1} d\tau+\int_{t_2-1}^{t_2}(t_2-\tau)^{\frac{1}{l_2}
-1}d\tau\Biggr)\Biggr)dt_2\Biggr)dt_1\\
<(b_2-a_2)(b_1-a_1)\Biggl[C_2 C_3\Biggl(\frac{b_1+a_1}{2}-1-a_1+l_1\Biggr)\\
+C_1 C_4\Biggl(\frac{b_2+a_2}{2}-1-a_2+l_2\Biggr)\Biggr]<\infty.
\end{multline*}
Hence, we can use Fubini's theorem to change
the order of integration:
\begin{multline*}
\int_{a_1}^{b_1}\int_{a_2}^{b_2}\Biggl[g(t_1,t_2){_{a_1}}
\textsl{I}^{\alpha_1(\cdot,\cdot)}_{t_1}\eta_1(t_1,t_2)\\
+f(t_1,t_2){_{a_2}}\textsl{I}^{\alpha_2(\cdot,\cdot)}_{t_2}
\eta_2(t_1,t_2)\Biggr]dt_2 dt_1\\
= \int_{a_1}^{b_1}\int_{a_2}^{b_2}\Biggl[
g(t_1,t_2)\int_{a_1}^{t_1}\frac{1}{\Gamma(\alpha_1(t_1,\tau))}\\
(t_1-\tau)^{\alpha_1(t_1,\tau)-1}\eta_1(\tau,t_2)d\tau\\
+f(t_1,t_2)\int_{a_2}^{t_2}\frac{1}{\Gamma(\alpha_2(t_2,\tau))}\\
(t_2-\tau)^{\alpha_2(t_2,\tau)-1}\eta_2(t_1,\tau)d\tau \Biggr]dt_2 dt_1\\
= \int_{a_1}^{b_1}\int_{a_2}^{b_2}\biggl[\eta_1(\tau,t_2)
\int_{\tau}^{b_1}\frac{1}{\Gamma(\alpha_1(t_1,\tau))}\\
(t_1-\tau)^{\alpha_1(t_1,\tau)-1}g(t_1,t_2)dt_1 \Biggr]dt_2 d\tau\\
+\int_{a_1}^{b_1}\int_{a_2}^{b_2}\biggl[\eta_2(t_1,\tau)
\int_{\tau}^{b_2}\frac{1}{\Gamma(\alpha_2(t_2,\tau))}\\
(t_2-\tau)^{\alpha_2(t_2,\tau)-1}f(t_1,t_2)dt_2 \Biggr]d\tau dt_1\\
=\int_{a_1}^{b_1}\int_{a_2}^{b_2}\eta_1(\tau,t_2){_{\tau}}
\textsl{I}^{\alpha_1(\cdot,\cdot)}_{b_1}g(\tau,t_2) dt_2 d\tau \\
+ \int_{a_1}^{b_1}\int_{a_2}^{b_2}\eta_2(t_1,\tau){_{\tau}}
\textsl{I}^{\alpha_2(\cdot,\cdot)}_{b_2}f(t_1,\tau) d\tau dt_1.
\end{multline*}
\end{proof}

We are now in conditions to state and prove
the Green theorem for derivatives of variable fractional order.

\begin{theorem}
\label{thm:GTD}
Let $0<\alpha_i(t_i,\tau)<1-\frac{1}{l_i}$ for all $t_i, \tau \in [a_i,b_i]$,
where $l_i\in\mathbb{N}$, $i=1,2$, are greater or equal than two.
If $f,g,\eta\in C^1\left(\Delta_2\right)$ and
${_{t_1}}\textsl{I}^{1-\alpha_1(\cdot,\cdot)}_{b_1}g,
{_{t_2}}\textsl{I}^{1-\alpha_2(\cdot,\cdot)}_{b_2}f\in C^1\left(\Delta_2\right)$,
then the following formula holds:
\begin{multline*}
\int_{a_1}^{b_1}\int_{a_2}^{b_2}\left[
g(\mathbf{t}){^{C}_{a_1}}\textsl{D}^{\alpha_1(\cdot,\cdot)}_{t_1}\eta(\mathbf{t})+f(\mathbf{t})
{^{C}_{a_2}}\textsl{D}^{\alpha_2(\cdot,\cdot)}_{t_2}\eta(\mathbf{t})\right]dt_2 dt_1\\
=\int_{a_1}^{b_1}\int_{a_2}^{b_2}\eta(\mathbf{t})\left[
{_{t_1}}\textsl{D}^{\alpha_1(\cdot,\cdot)}_{b_1}g(\mathbf{t})
+{_{t_2}}\textsl{D}^{\alpha_2(\cdot,\cdot)}_{b_2}f(\mathbf{t})\right]dt_2 dt_1\\
+\oint_{\partial\Delta_2}\eta(\mathbf{t})\left[{_{t_1}}\textsl{I}^{1-\alpha_1(\cdot,\cdot)}_{b_1}
g(\mathbf{t})dt_2-{_{t_2}}\textsl{I}^{1-\alpha_2(\cdot,\cdot)}_{b_2}f(\mathbf{t})dt_1\right].
\end{multline*}
\end{theorem}

\begin{proof}
By definition of Caputo partial derivative of variable fractional order,
Theorem~\ref{theorem:GRI}, and the standard Green's theorem, one has
\begin{multline*}
\int_{a_1}^{b_1}\int_{a_2}^{b_2}\Biggl[
g(\mathbf{t}){^{C}_{a_1}}\textsl{D}^{\alpha_1(\cdot,\cdot)}_{t_1}\eta(\mathbf{t})
+f(\mathbf{t}) {^{C}_{a_2}}\textsl{D}^{\alpha_2(\cdot,\cdot)}_{t_2}\eta(\mathbf{t})\Biggr]dt_2 dt_1\\
=\int_{a_1}^{b_1}\int_{a_2}^{b_2}\Biggl[g(\mathbf{t}){_{a_1}}\textsl{I}^{1
-\alpha_1(\cdot,\cdot)}_{t_1}\frac{\partial}{\partial t_1}\eta(\mathbf{t})\\
+f(\mathbf{t}) {_{a_2}}\textsl{I}^{1-\alpha_2(\cdot,\cdot)}_{t_2}
\frac{\partial}{\partial t_2}\eta(\mathbf{t})\Biggr]dt_2 dt_1\\
=\int_{a_1}^{b_1}\int_{a_2}^{b_2}\Biggl[\frac{\partial}{\partial t_1}
\eta(\mathbf{t}){_{t_1}}\textsl{I}^{1-\alpha_1(\cdot,\cdot)}_{b_1}g(\mathbf{t})\\
+\frac{\partial}{\partial t_2}\eta(\mathbf{t}) {_{t_2}}\textsl{I}^{1
-\alpha_2(\cdot,\cdot)}_{b_2}f(\mathbf{t})\Biggr]dt_2 dt_1\\
=-\int_{a_1}^{b_1}\int_{a_2}^{b_2}\eta(\mathbf{t})\Biggl[
\frac{\partial}{\partial t_1}{_{t_1}}\textsl{I}^{1-\alpha_1(\cdot,\cdot)}_{b_1}g(\mathbf{t})\\
+\frac{\partial}{\partial t_2}{_{t_2}}\textsl{I}^{1-\alpha_2(\cdot,\cdot)}_{b_2}f(\mathbf{t})\Biggr]dt_2 dt_1\\
+\oint_{\partial\Delta_2}\eta(\mathbf{t})\left[{_{t_1}}\textsl{I}^{1-\alpha_1(\cdot,\cdot)}_{b_1}
g(\mathbf{t})dt_2-{_{t_2}}\textsl{I}^{1-\alpha_2(\cdot,\cdot)}_{b_2}f(\mathbf{t})dt_1\right]\\
=\int_{a_1}^{b_1}\int_{a_2}^{b_2}\eta(\mathbf{t})\left[{_{t_1}}\textsl{D}^{\alpha_1(\cdot,\cdot)}_{b_1}
g(\mathbf{t})+{_{t_2}}\textsl{D}^{\alpha_2(\cdot,\cdot)}_{b_2}f(\mathbf{t})\right]dt_2 dt_1\\
+\oint_{\partial\Delta_2}\eta(\mathbf{t})\left[{_{t_1}}\textsl{I}^{1-\alpha_1(\cdot,\cdot)}_{b_1}
g(\mathbf{t})dt_2-{_{t_2}}\textsl{I}^{1-\alpha_2(\cdot,\cdot)}_{b_2}f(\mathbf{t})dt_1\right].
\end{multline*}
\end{proof}


\section{VARIABLE ORDER FRACTIONAL CALCULUS OF VARIATIONS FOR DOUBLE INTEGRALS}
\label{sec:VOFVC}

Let $\alpha_i(t_i,\tau)$, $i=1,2$,
satisfy the assumptions of Theorem~\ref{thm:GTD}.
We consider the following problem:
\begin{problem}
\label{problem:Fundamental}
Find a function $u=u(\mathbf{t})$ for which
the fractional variational functional
\begin{equation*}
\mathcal{J}[u]
=\int_{a_1}^{b_1}\int_{a_2}^{b_2}
L\left(\mathbf{t},u(\mathbf{t}),{^{C}_{a_1}}\textsl{D}^{\alpha_1(\cdot,\cdot)}_{t_1}
u(\mathbf{t}),{^{C}_{a_2}}\textsl{D}^{\alpha_2(\cdot,\cdot)}_{t_2}u(\mathbf{t})\right)dt_2 dt_1
\end{equation*}
subject to the boundary condition
\begin{equation}\label{eq:FundBound}
\left.u(\mathbf{t})\right|_{\partial \Delta_2}=\psi(\mathbf{t}),
\end{equation}
where $\psi:\partial \Delta_2 \rightarrow \mathbb{R}$ is a given function, attains an extremum.
\end{problem}
We assume that $L\in C^1\left(\Delta_2\times \mathbb{R}^3;\mathbb{R}\right)$;
$t\mapsto\partial_{i+2}L$ is continuously differentiable, has continuously differentiable integral
${_{t_i}}\textsl{I}^{1-\alpha_i(\cdot,\cdot)}_{b_i}$, and continuous derivative
${_{t_i}}\textsl{D}^{\alpha_i(\cdot,\cdot)}_{b_i}$, $i=1,2$.
For simplicity of notation, we introduce the following operator:
\begin{displaymath}
\left\{u,\alpha_1,\alpha_2\right\}(\mathbf{t})
:=\left(\mathbf{t},u(\mathbf{t}),{^{C}_{a_1}}\textsl{D}^{\alpha_1(\cdot,\cdot)}_{t_1}
u(\mathbf{t}),{^{C}_{a_2}}\textsl{D}^{\alpha_2(\cdot,\cdot)}_{t_2}u(\mathbf{t})\right).
\end{displaymath}

A typical example for the cost
functional $\mathcal{J}$ appears when one considers
the shape of a string during the course of the vibration
(\textrm{cf.} \cite{MyID:182}):
\begin{multline*}
\mathcal{J}[u] =\int_{t}^{x}\int_{0}^{L}
\left[ \sigma(x) \left({^{C}_{a_2}}\textsl{D}^{\alpha_2(\cdot,\cdot)}_{x}
u(x,t)\right)^2\right.\\
\left.-\tau \left({^{C}_{a_1}}\textsl{D}^{\alpha_1(\cdot,\cdot)}_{t}
u(x,t)\right)^2 \right] dx dt,
\end{multline*}
where $\tau$ is the constant tension and
$\sigma(x)$ is the string density.

\begin{definition}
A continuously differentiable function $u$ is said to be
admissible for Problem~\ref{problem:Fundamental}
if ${^{C}_{a_i}}\textsl{D}^{\alpha_i(\cdot,\cdot)}_{t_i}u$ exist and are continuous on the
rectangle $\Delta_2$, $i=1,2$, and $u$ satisfies the boundary condition \eqref{eq:FundBound}.
\end{definition}

\begin{theorem}
\label{thm:ELCaputo}
If $u$ is a solution to Problem~\ref{problem:Fundamental},
then $u$ satisfies the Euler--Lagrange equation
\begin{multline}
\label{eq:eqELCaputo}
\partial_{2} L\left\{u,\alpha_1,\alpha_2\right\}(\mathbf{t})
+{_{t_1}}\textsl{D}^{\alpha_1(\cdot,\cdot)}_{b_1}\partial_{3}
L\left\{u,\alpha_1,\alpha_2\right\}(\mathbf{t})\\
+{_{t_2}}\textsl{D}^{\alpha_2(\cdot,\cdot)}_{b_2}\partial_{4}
L\left\{u,\alpha_1,\alpha_2\right\}(\mathbf{t})=0,
\quad \mathbf{t}\in\Delta_2.
\end{multline}
\end{theorem}

\begin{proof}
Suppose that $u$ is an extremizer for $\mathcal{J}$. Consider $\eta \in C^1(\Delta_2;\mathbb{R})$
such that ${^{C}_{a_i}}\textsl{D}^{\alpha_i(\cdot,\cdot)}_{t_i}\eta \in C(\Delta_2;\mathbb{R})$,
$i=1,2$, and $\left.\eta(\mathbf{t})\right|_{\partial \Delta_2}\equiv0$.
We imbed $u$ in the one-parameter family of functions
$\{\hat{u}=u+\varepsilon\eta:|\varepsilon|<\varepsilon_0,\varepsilon_0>0\}$.
Define
\begin{multline*}
J(\varepsilon)=\mathcal{J}[\hat{u}]\\
=\int_{a_1}^{b_1}\int_{a_2}^{b_2} L\left(\mathbf{t},\hat{u}(\mathbf{t}),
{^{C}_{a_1}}\textsl{D}^{\alpha_1(\cdot,\cdot)}_{t_1}\hat{u}(\mathbf{t}),
{^{C}_{a_2}}\textsl{D}^{\alpha_2(\cdot,\cdot)}_{t_2}\hat{u}(\mathbf{t})\right)dt_2 dt_1.
\end{multline*}
Then, a necessary condition for $u$ to be an extremizer for $\mathcal{J}$ is given by
\begin{multline*}
\left.\frac{d J}{d\varepsilon}\right|_{\varepsilon=0}=0
\Leftrightarrow\int_{a_1}^{b_1}\int_{a_2}^{b_2}
\Biggl(\partial_{2} L\left\{u,\alpha_1,\alpha_2\right\}(\mathbf{t}) \cdot \eta(\mathbf{t})\\
+\partial_{3} L\left\{u,\alpha_1,\alpha_2\right\}(\mathbf{t})
\cdot {^{C}_{a_1}}\textsl{D}^{\alpha_1(\cdot,\cdot)}_{t_1}\eta(\mathbf{t})\\
+\partial_{4} L\left\{u,\alpha_1,\alpha_2\right\}(\mathbf{t})
\cdot {^{C}_{a_2}}\textsl{D}^{\alpha_2(\cdot,\cdot)}_{t_2}\eta(\mathbf{t}) \Biggr)dt_2dt_1 = 0.
\end{multline*}
By Theorem~\ref{thm:GTD}, and since $\left.\eta(\mathbf{t})\right|_{\partial \Delta_2}\equiv0$, one has
\begin{multline*}
\int_{a_1}^{b_1}\int_{a_2}^{b_2}
\Biggl(\partial_{3} L\left\{u,\alpha_1,\alpha_2\right\}(\mathbf{t})
\cdot {^{C}_{a_1}}\textsl{D}^{\alpha_1(\cdot,\cdot)}_{t_1}\eta(\mathbf{t})\\
+\partial_{4} L\left\{u,\alpha_1,\alpha_2\right\}(\mathbf{t})
\cdot {^{C}_{a_2}}\textsl{D}^{\alpha_2(\cdot,\cdot)}_{t_2}\eta(\mathbf{t})\Biggr) dt_2 dt_1\\
=\int_{a_1}^{b_1}\int_{a_2}^{b_2}\eta(\mathbf{t})\Biggl(
{_{t_1}}\textsl{D}^{\alpha_1(\cdot,\cdot)}_{b_1}\partial_{3}
L\left\{u,\alpha_1,\alpha_2\right\}(\mathbf{t})\\
+{_{t_2}}\textsl{D}^{\alpha_2(\cdot,\cdot)}_{b_2}\partial_{4}
L\left\{u,\alpha_1,\alpha_2\right\}(\mathbf{t})\Biggr)dt_2 dt_1.
\end{multline*}
Therefore,
\begin{multline*}
\int_{a_1}^{b_1}\int_{a_2}^{b_2}\eta(\mathbf{t})\Biggl(\partial_{2}
L\left\{u,\alpha_1,\alpha_2\right\}(\mathbf{t})\\
+{_{t_1}}\textsl{D}^{\alpha_1(\cdot,\cdot)}_{b_1}\partial_{3}
L\left\{u,\alpha_1,\alpha_2\right\}(\mathbf{t})\\
+{_{t_2}}\textsl{D}^{\alpha_2(\cdot,\cdot)}_{b_2}\partial_{4}
L\left\{u,\alpha_1,\alpha_2\right\}(\mathbf{t})\Biggr)dt_2 dt_1=0.
\end{multline*}
Condition \eqref{eq:eqELCaputo} follows from the
fundamental lemma of the calculus of variations.
\end{proof}


\section{CONCLUSIONS}
\label{sec:conc}

Recently, the variable order fractional calculus has
provided new insights into rich applications in diverse
fields such as physics, cyber-physical systems, signal
processing, and mean field games
\cite{ref:giv:r:2,ref:giv:r:4,ref:giv:r:1,ref:giv:r:3}.
In this article a multidimensional integration by parts formula
for partial integrals of variable fractional order
(Theorem~\ref{theorem:GRI}) and a Green type theorem
with derivatives and integrals of variable
fractional order (Theorem~\ref{thm:GTD}) are proved.
These theorems are then used to obtain Euler--Lagrange
type equations for the minimization
of a functional involving derivatives of variable
fractional order (Theorem~\ref{thm:ELCaputo}).
Our results generalize the recent one-dimensional theory
of the fractional calculus of variations of variable order
\cite{MyID:218} to the two-dimensional case, \textrm{i.e.},
for fractional variational problems with double integrals.


\section*{ACKNOWLEDGMENTS}

Work supported by FEDER funds through
COMPETE (Operational Programme Factors of Competitiveness)
and by Portuguese funds through the
{\it Center for Research and Development
in Mathematics and Applications} (University of Aveiro)
and the Portuguese Foundation for Science and Technology
(FCT), within project PEst-C/MAT/UI4106/2011
with COMPETE number FCOMP-01-0124-FEDER-022690.
Odzijewicz was also supported by FCT through the Ph.D. fellowship
SFRH/BD/33865/2009; Malinowska by Bia{\l}ystok
University of Technology grant S/WI/02/2011;
and Torres by FCT through the Portugal--Austin
cooperation project UTAustin/MAT/0057/2008.
The authors are grateful to three anonymous referees
for valuable remarks and comments which
significantly contributed to the quality of the paper.



\bigskip

{\bf This is a preprint of a paper whose final and definite form will be published in:
51st IEEE Conference on Decision and Control, December 10-13, 2012, Maui, Hawaii, USA.
Article Source/Identifier: PLZ-CDC12.1240.d4462b33.
Submitted 07-March-2012; accepted 17-July-2012.}



\begin{thebibliography}{99}

\bibitem{Abel}
N. H. Abel,
{E}uvres completes de Niels Henrik Abel.
Christiana: Imprimerie de Grondahl and Son; New York and London:
Johnson Reprint Corporation. VIII, 621 pp., 1965.

\bibitem{MR2762638}
O. P. Agrawal, O. Defterli\ and\ D. Baleanu,
Fractional optimal control problems with several state and control variables,
J. Vib. Control {\bf 16} (2010), no.~13, 1967--1976.

\bibitem{MyID:182}
R. Almeida, A. B. Malinowska and D. F. M. Torres,
A fractional calculus of variations for multiple
integrals with application to vibrating string,
J. Math. Phys. vol.~51, 2010, 033503, 12~pp.
{\tt arXiv:1001.2722}

\bibitem{MyID:209}
R. Almeida, A. B. Malinowska and D. F. M. Torres,
Fractional Euler-Lagrange differential equations via Caputo derivatives,
In: Fractional Dynamics and Control, Springer New York, 2012, Part~2, 109--118.
{\tt arXiv:1109.0658}

\bibitem{Shakoor:01}
R. Almeida, S. Pooseh and D. F. M. Torres,
Fractional variational problems depending on indefinite integrals,
Nonlinear Anal. vol.~75, 2012, no.~3, 1009--1025.
{\tt arXiv:1102.3360}

\bibitem{Atanackovic1}
T. M. Atanackovic and S. Pilipovic,
Hamilton's principle with variable order fractional derivatives,
Fract. Calc. Appl. Anal. vol.~14, 2011, 94--109.

\bibitem{ref:giv:r:2}
P. Bogdan and R. Marculescu,
Towards a Science of Cyber-Physical Systems Design,
Cyber-Physical Systems (ICCPS), 2011 IEEE/ACM
International Conference, 12-14 April 2011, pp.~99--108.

\bibitem{ref:giv:r:4}
P. Bogdan, R. Marculescu,
A fractional calculus approach
to modeling fractal dynamic games,
Decision and Control and European Control Conference (CDC-ECC),
2011 50th IEEE Conference, 12-15 Dec. 2011, pp.~255--260.

\bibitem{ref:giv:r:1}
A. V. Chechkin, R. Gorenflo and I. M. Sokolov,
Fractional diffusion in inhomogeneous media,
J. Phys. A: Math. Gen. 38, 2005.

\bibitem{Coimbra}
C. F. M. Coimbra,
Mechanics with variable-order differential operators,
Ann. Phys. vol.~12, 2003, 692–-703.

\bibitem{Cresson}
J. Cresson,
Fractional embedding of differential operators and Lagrangian systems,
J. Math. Phys. vol.~48, 2007, no.~3, 033504,~34 pp.
{\tt arXiv:math/0605752}

\bibitem{Cresson2}
J. Cresson,
Inverse problem of fractional calculus of variations
for partial differential equations,
Commun. Nonlinear Sci. Numer. Simul. vol.~15, 2010, no.~4, 987--996.

\bibitem{Diaz:Coimbra}
G. Diaz and C. F. M. Coimbra,
Nonlinear dynamics and control of a variable order oscillator
with application to the van der Pol equation,
Nonlinear Dynam. vol.~56, 2009, 145--157.

\bibitem{gastao1}
G. S. F. Frederico and D. F. M. Torres,
A formulation of Noether's theorem
for fractional problems of calculus of variations,
J. Math. Anal. Appl. vol.~334, 2007, 834--846.
{\tt arXiv:math/0701187}

\bibitem{MyID:089}
G. S. F. Frederico and D. F. M. Torres,
Fractional conservation laws in optimal control theory,
Nonlinear Dynam. vol.~53, 2008, no.~3, 215--222.
{\tt arXiv:0711.0609}

\bibitem{MR2905862}
Md.\ M. Hasan, X. W. Tangpong\ and\ O. P. Agrawal,
A formulation and numerical scheme for fractional optimal control
of cylindrical structures subjected to general initial conditions,
in {\it Fractional dynamics and control}, 3--17, Springer, New York, 2012.

\bibitem{Ivady}
P. Ivady,
A note on a gamma function inequality,
J. Math. Inequal. vol.~3, 2009, 227--236.

\bibitem{book:Kilbas}
A. A. Kilbas, H. M. Srivastava and J. J. Trujillo,
Theory and applications of fractional differential equations.
North-Holland Mathematics Studies, 204, Elsevier, Amsterdam, 2006.

\bibitem{book:Klimek}
M. Klimek,
On solutions of linear fractional differential equations of a variational type,
The Publishing Office of Czenstochowa University of Technology, Czestochowa, 2009.

\bibitem{Lorentzo}
C. F. Lorenzo and T. T. Hartley,
Variable order and distributed order fractional operators,
Nonlinear Dynam. vol.~29, 2002, 57--98.

\bibitem{MyID:208}
A. B. Malinowska and D. F. M. Torres,
Introduction to the fractional calculus of variations,
Imperial College Press, London
\& World Scientific Publishing, Singapore, 2012.

\bibitem{MyID:181}
D. Mozyrska and D. F. M. Torres,
Modified optimal energy and initial
memory of fractional continuous-time linear systems,
Signal Process. vol.~91, 2011, 379--385.
{\tt arXiv:1007.3946}

\bibitem{MyID:207}
T. Odzijewicz, A. B. Malinowska and D. F. M. Torres,
Fractional variational calculus with classical and combined Caputo derivatives,
Nonlinear Anal. vol.~75, 2011, 1507--1515.
{\tt arXiv:1101.2932}

\bibitem{FVC_Gen}
T. Odzijewicz, A. B. Malinowska and D. F. M. Torres,
Generalized fractional calculus with applications to the calculus of variations,
Comput. Math. Appl., 2012, DOI: 10.1016/j.camwa.2012.01.073.
{\tt arXiv:1201.5747}

\bibitem{MyID:227}
T. Odzijewicz, A. B. Malinowska and D. F. M. Torres,
Fractional calculus of variations in terms of a generalized
fractional integral with applications to Physics,
Abstr. Appl. Anal. vol.~2012, 2012, Article ID 871912, 24 pages.
{\tt arXiv:1203.1961}

\bibitem{MyID:218}
T. Odzijewicz, A. B. Malinowska and D. F. M. Torres,
Fractional variational calculus of variable order.
In: Advances in Harmonic Analysis and Operator Theory,
The Stefan Samko Anniversary Volume
(Eds: A. Almeida, L. Castro, F.-O. Speck),
Operator Theory: Advances and Applications, Birkh\"auser Verlag, in press.
{\tt arXiv:1110.4141}

\bibitem{MyID:236}
T. Odzijewicz, A. B. Malinowska and D. F. M. Torres,
Green's theorem for generalized fractional derivatives,
Proceedings of FDA'2012, The Fifth Symposium
on Fractional Differentiation and its Applications,
May 14-17, 2012, Hohai University, Nanjing, China.
Editors: W. Chen, H.-G. Sun and D. Baleanu.
Paper \#084, 2012.
{\tt arXiv:1205.4851}

\bibitem{tatiana}
T. Odzijewicz and D. F. M. Torres,
Fractional calculus of variations for double integrals,
Balkan J. Geom. Appl. vol.~16, no.~2, 2011, 102--113.
{\tt arXiv:1102.1337}

\bibitem{book:Podlubny}
I. Podlubny,
Fractional differential equations.
Mathematics in Science and Engineering,
198, Academic Press, San Diego, CA, 1999.

\bibitem{Ramirez:Coimbra1}
L. E. S. Ramirez and C. F. M. Coimbra,
On the selection and meaning of variable order operators for dynamic modeling,
Int. J. Differ. Equ. vol.~2010, 2010, Art. ID 846107, 16~pp.

\bibitem{Ramirez:Coimbra2}
L. E. S. Ramirez and C. F. M. Coimbra,
On the variable order dynamics of the nonlinear wake caused by a sedimenting particle,
Phys. D vol.~240, 2011, 1111--1118.

\bibitem{CD:Riewe:1996}
F. Riewe,
Nonconservative Lagrangian and Hamiltonian mechanics,
Phys. Rev. E (3) vol.~53, 1996, 1890--1899.

\bibitem{CD:Riewe:1997}
F. Riewe,
Mechanics with fractional derivatives,
Phys. Rev. E (3) vol.~55, 1997, 3581--3592.

\bibitem{Samko:Ross1}
B. Ross and S. G. Samko,
Fractional integration operator of a variable order in the Holder spaces $H^{\lambda(x)}$,
Internat. J. Math. Math. Sci. vol.~18, 1995, 777--788.

\bibitem{Samko}
S. G. Samko,
Fractional integration and differentiation of variable order,
Anal. Math. vol.~21, 1995, 213--236.

\bibitem{book:Samko}
S. G. Samko, A. A. Kilbas and O. I. Marichev,
Fractional integrals and derivative.
Translated from the 1987 Russian original,
Gordon and Breach, Yverdon, 1993.

\bibitem{Samko:Ross2}
S. G. Samko and B. Ross,
Integration and differentiation to a variable fractional order,
Integral Transform. Spec. Funct. vol.~1, 1993, 277--300.

\bibitem{ref:giv:r:3}
H. Sheng, Y.-Q. Chen and T.-S. Qiu,
Fractional Processes and Fractional-Order Signal Processing:
Techniques and Applications, Springer, 2011.

\end{thebibliography}
\end{document}